\documentclass[reqno, 12pt,a4letter]{amsart}
\usepackage{amsmath,amsxtra,amssymb,latexsym, amscd,amsthm}
\usepackage{graphicx,color}
\usepackage[utf8]{inputenc}
\usepackage[mathscr]{euscript}
\usepackage{mathrsfs}
\usepackage[english]{babel}
\usepackage{enumerate}
\newtheorem {theorem}{Theorem}[section]
\newtheorem {corollary}{Corollary}[section]
\newtheorem {lemma}{Lemma}[section]
\newtheorem {proposition}{Proposition}[section]

\theoremstyle{definition}
\newtheorem{definition}{Definition}[section]
\newtheorem{example}{Example}[section]
\newtheorem{remark}{Remark}[section]

\def\ar{a\kern-.370em\raise.16ex\hbox{\char95\kern-0.53ex\char'47}\kern.05em}
\def\ees{{\accent"5E e}\kern-.385em\raise.2ex\hbox{\char'23}\kern-.08em}
\def\eex{{\accent"5E e}\kern-.470em\raise.3ex\hbox{\char'176}}
\def\AR{A\kern-.46em\raise.80ex\hbox{\char95\kern-0.53ex\char'47}\kern.13em}
\def\EES{{\accent"5E E}\kern-.5em\raise.8ex\hbox{\char'23 }}
\def\EEX{{\accent"5E E}\kern-.60em\raise.9ex\hbox{\char'176}\kern.1em}
\def\ow{o\kern-.42em\raise.82ex\hbox{
  \vrule width .12em height .0ex depth .075ex \kern-0.16em \char'56}\kern-.07em}
\def\OW{O\kern-.460em\raise1.36ex\hbox{
\vrule width .13em height .0ex depth .075ex \kern-0.16em \char'56}\kern-.07em}
\def\UW{U\kern-.42em\raise1.36ex\hbox{
\vrule width .13em height .0ex depth .075ex \kern-0.16em \char'56}\kern-.07em}
\def\DD{D\kern-.7em\raise0.4ex\hbox{\char '55}\kern.33em}

\title{Openness,  H\"older metric regularity and H\"older continuity properties of semialgebraic set-valued~maps}

\author{Jae Hyoung Lee$^\dag$}
\address[Jae Hyoung Lee]{Department of Applied Mathematics, Pukyong National University, Busan 48513, Republic of Korea}
\email{mc7558@naver.com}

\author{TI\EES N-S\OW N PH\d{A}M$^\ddag$}
\address[Ti\ees n-S\ow n Ph\d{a}m]{Department of Mathematics, University of Dalat, 1 Phu Dong Thien Vuong, Dalat, Vietnam}
\email{sonpt@dlu.edu.vn}

\thanks{$^{\dag}$The first author is supported by the National Research Foundation of Korea (NRF) Grant funded by the Korea government (MSIP) (NRF-2018R1C1B6001842)}
\thanks{$^{\ddag}$The second author is supported by Vietnam National Foundation for Science and Technology Development (NAFOSTED), grant 101.04-2019.302.}

\subjclass{90C31, 49J40, 49J52, 49J53, 14P10}

\keywords{openness, H\"older metric (sub)regularity, H\"older continuity, \L ojasiewicz inequality, variational inequality, semialgebraicity}

\date{ \today}

\begin{document}

\begin{abstract}
Given a semialgebraic set-valued map $F \colon \mathbb{R}^n \rightrightarrows \mathbb{R}^m$ with closed graph, 
we show that the map $F$ is H\"older metrically subregular and that the following conditions are equivalent:
\begin{enumerate}[{\rm (i)}]
\item $F$ is an open map from its domain into its range and the range of $F$ is locally closed;
\item the map $F$ is H\"older metrically regular;
\item the inverse map $F^{-1}$ is pseudo-H\"older continuous;
\item the inverse map $F^{-1}$ is lower pseudo-H\"older continuous.
\end{enumerate}

An application, via Robinson's normal map formulation, leads to the following result in the context of semialgebraic variational inequalities: if the solution map (as a map of the parameter vector) is lower semicontinuous then the solution map is finite and pseudo-H\"older continuous. In particular, we obtain a negative answer to a question 
mentioned in the paper of Dontchev and Rockafellar \cite{Dontchev1996}. 

As a byproduct, we show that for a (not necessarily semialgebraic) continuous single-valued map from $\mathbb{R}^n$ to $\mathbb{R},$ the openness and the  non-extremality are equivalent. This fact improves the main result of P\"uhn \cite{Puhl1998}, which requires the convexity of the map in question.
\end{abstract}

\maketitle

\section{Introduction}

This paper concerns with the openness, the H\"older metric (sub)regularity and the H\"older continuity properties of set-valued maps between finite dimensional vector spaces.
These notions have been recognized to be important in modern variational analysis and optimization and have extensively studied; for more details, we refer the interested reader to the books \cite{Aubin1990, Bonnas2000, Dontchev2009, Ioffe2017, Klatte2002, Mordukhovich2006, Rockafellar1998, Schirotzek2007}, the papers \cite{Aubin1984, Borwein1988, Chuong2016, Frankowska1990, Frankowska2012, Gfrerer2016, Huynh2008, Ioffe2013, Kruger2015, Li2012, Mordukhovich2015, Penot1989, Rockafellar1985, Yen2008, Zheng2015} and to the references contained therein. We offer here a new approach that allows us to derive relations between  the openness, the H\"older metric regularity and the H\"older continuity properties for the class of semialgebraic maps. This approach is mainly based on tools of semialgebraic geometry that extends known results. 

More precisely, let $F \colon \mathbb{R}^n \rightrightarrows \mathbb{R}^m$ be a set-valued map. For simplicity in presentation, we assume in this section that the domain and the range of $F$ are open sets in $\mathbb{R}^n$ and $\mathbb{R}^m,$ respectively.

Penot \cite{Penot1989} showed that the openness at a linear rate of $F,$ the metric regularity of $F,$ and the pseudo-Lipschitz continuity (known also as the Aubin property or the Lipschitz-like continuity) of the inverse map $F^{-1}$ are equivalent; detailed proofs of this result can be found also in \cite[Theorem~3.2]{Mordukhovich2018}.
More generally, thanks to Borwein and Zhuang \cite{Borwein1988} (see also \cite{Ioffe2013, Yen2008}), it is well-known that the openness at an order $p > 0$ rate of $F,$ the metric regularity property of order $1/p$ of $F,$ and the pseudo-H\"older continuity of order $1/p$ of $F^{-1}$ are equivalent.

Since the openness at a linear/positive-order rate is stronger than the openness, the pseudo-Lipschitz/H\"older continuity of $F^{-1}$ implies the openness of $F$ but the converse does not hold in general (see, e.g., Example~\ref{Example5} below).

On the other hand, Gowda and Sznajder \cite{Gowda1996} showed that if the map $F$ is polyhedral (i.e., its graph is a finite union of polyhedra) and the range of $F$ is convex, then the map $F$ is open if and only if the inverse map $F^{-1}$ is Lipschitz continuous, equivalently, locally lower Lipschitz continuous; see also Remark~\ref{Remark1} below. As we can see the polyhedrality of the map $F$ plays a crucial role in establishing these equivalences. 

In nonlinear programming, the situation is quite different: maps in general are not polyhedral, and so relationships between the openness, metric regularity and continuity properties are unclear. On the other hand, since polyhedral maps form a very special subclass of semialgebraic maps, it is natural to study these properties for maps which are semialgebraic.

In this paper, assuming $F$ is a semialgebraic set-valued map with closed graph, we show that  $F$ is H\"older metrically subregular and that the following conditions are equivalent (terminology will be explained later):
\begin{enumerate}[{\rm (i)}]
\item the map $F$ is open (i.e., $F$ carries open sets into open sets);
\item the map $F$ is H\"older metrically regular;
\item the inverse map $F^{-1}$ is pseudo-H\"older continuous;
\item the inverse map $F^{-1}$ is lower pseudo-H\"older continuous.
\end{enumerate}

It should be noted that these equivalences are no longer true in general if we replace the word {\em H\"older} by {\em Lipschitz} or 
if we drop the assumption that $F$ is semialgebraic (see examples in Section~\ref{Section5}).

The main point in the proof of the above equivalences is to show the implication (i)~$\Rightarrow$~(ii). To this end, we first prove that $F$ is open if and only if the inverse images $F^{-1}(y)$ vary continuously (see Lemma~\ref{Lemma31}). This fact, together with a generalization of the \L ojasiewicz inequality (see Lemma~\ref{Lemma32}), implies that if $F$ is open then it is H\"older metrically regular.

As a byproduct of this study, we show that for a (not necessarily semialgebraic) continuous single-valued map from $\mathbb{R}^n$ to $\mathbb{R},$ the openness and the  non-extremality are equivalent. This improves the main result of P\"uhn \cite{Puhl1998} (see also \cite[Proposition~5.11]{Durea2017}), which requires the convexity of the map in question.

At this point we would like to mention that there are necessary and sufficient conditions for a single-valued map from $\mathbb{R}^n$ to itself to be open. This was done 
in \cite{Scholtes2012} by Scholtes for piecewise affine maps, in \cite{Gamboa1996} by Gamboa and Ronga for polynomial maps,  in \cite{Hirsch2002} by Hirsch in the analytic setting, and recently in \cite{Denkowski2017} by Denkowski and Loeb for  subanalytic (or definable in some o-minimal structure) maps of class $C^1.$

As an application of our main theorem, we study the variational inequality problem in which a point $x \in C$ is sought such that
\begin{equation} \label{VI}
\langle p + f(x), x' - x \rangle \ \ge \  0 \quad \textrm{ for all } \quad x' \in C, \tag{VI}
\end{equation}
where $p \in \mathbb{R}^n$ is a parameter vector, $f \colon \mathbb{R}^n \rightarrow \mathbb{R}^n$ is a map and $C \subset \mathbb{R}^n$ is a nonempty closed convex set.

For each $p \in \mathbb{R}^n$ let $\mathscr{S}(p)$ be the (possibly empty) set of solutions of the variational inequality~\eqref{VI}. We concern ourselves with continuity properties of the set-valued map $\mathscr{S} \colon \mathbb{R}^n \rightrightarrows \mathbb{R}^n.$
Specifically we are interested in the circumstances under which the solution map $\mathscr{S}$ is finite (i.e., $\mathscr{S}(p)$ is a finite set for all $p \in \mathbb{R}^n).$

We first assume that the map $f$ is of class $C^1$ and the set $C$ is polyhedral. Let $\mathscr{L} \colon \mathbb{R}^n \rightrightarrows \mathbb{R}^n$ denote the solution map corresponding to a linearized version (at a given point) of the variational inequality~\eqref{VI}.  Thanks to Robinson \cite{Robinson1980}, it is well-known that  the local single valuedness and the Lipschitz continuity of the map $\mathscr{L}$ (a property called ``strong regularity'') implies the same for the map $\mathscr{S}.$ On the other hand, Dontchev and Rockafellar \cite{Dontchev1996} (see also \cite{Ioffe2018}) showed that the converse is also true and that these equivalent conditions are characterized in terms of the so-called ``critical face" condition and in other ways. Furthermore, the lower semicontinuity of the map $\mathscr{L}$ entails the local single valuedness and the Lipschitz continuity of the map $\mathscr{L}.$ The latter result can be compared with the well-known fact that a monotone map has to be single-valued and continuous wherever it is lower semicontinuous. In the introduction section of the paper \cite{Dontchev1996}, the authors note that they do not know {\em whether the lower semicontinuity of the map $\mathscr{S}$ ensures the local single-valuedness and the Lipschitz continuity of the map $\mathscr{S}.$}

Motivated by the aforementioned works, assuming $f$ is a continuous semialgebraic map and $C$ is a closed convex semialgebraic set, we will show that the following conditions are equivalent:
\begin{enumerate}[{\rm (i)}]
\item the map $\mathscr{S}$ is lower semicontinuous;
\item the map $\mathscr{S}$ is pseudo-H\"older continuous;
\item the map $\mathscr{S}$ is lower pseudo-H\"older continuous;
\item the inverse map $\mathscr{S}^{-1}$ is open;
\item the inverse map $\mathscr{S}^{-1}$ is H\"older metrically regular;
\item the {\em (Robinson) normal map} $\mathscr{F}$ associated with the variational inequality~\eqref{VI} is open;
\item the map $\mathscr{F}$ is H\"older metrically regular;
\item the inverse map $\mathscr{F}^{-1}$ is pseudo-H\"older continuous;
\item the inverse map $\mathscr{F}^{-1}$ is lower pseudo-H\"older continuous.
\end{enumerate}
Specifically, these equivalences, combined with some facts of semialgebraic geometry, will imply that the lower semicontinuity of the map $\mathscr{S}$ is already enough to guarantee the finiteness and the pseudo-H\"older continuity of the map $\mathscr{S}.$

Again, we can not replace {\em H\"older} by {\em Lipschitz}, and furthermore, the lower semicontinuity of the map $\mathscr{S}$ does {\em not} ensure its local single-valuedness; see Example~\ref{Example53} below. So we get a negative answer to the question of Dontchev and Rockafellar mentioned above. 

We would like to remark that it is possible to obtain the local versions of the (global) results presented in this paper; alterations needed for the local versions are more or less straightforward. Furthermore, although the results still hold for maps definable in some o-minimal structure (see \cite{Dries1996} for more on the subject), we prefer to use semialgebraic maps for simplicity.

The rest of this paper is organized as follows. Section~\ref{SectionPreliminary} contains some preliminaries from semialgebraic geometry widely used in the proofs of the main results, which will be given in Sections~\ref{Section3} and \ref{Section4}. Finally, several remarks/examples are provided in Section~\ref{Section5}.

\section{Preliminaries} \label{SectionPreliminary}

Throughout this work we shall consider the Euclidean vector space ${\Bbb R}^n$ endowed with its canonical scalar product $\langle \cdot, \cdot \rangle$ and we shall denote its associated norm $\| \cdot \|.$ The closure and interior of a set $X \subset \mathbb{R}^n$ are denoted by $\mathrm{cl}X$ and $\mathrm{int} X,$ respectively.
We denote by $\mathbb{B}_{R}({x})$ the closed ball centered at ${x}$ with radius $R$ and by $\mathbb{B}$ the
the closed unit ball. As usual, $\mathrm{dist}(x, X)$ denotes the Euclidean distance from $x \in \mathbb{R}^n$ to $X \subset \mathbb{R}^n,$ i.e.,
\begin{eqnarray*}
\mathrm{dist}(x, X) &:=& \inf\{\|x - y \| \ | \ y \in X \}.
\end{eqnarray*}
(We adopt the convention that the distance from a point to the empty set is positive infinity.) A subset $X$ of $\mathbb{R}^n$ is {\em locally closed} if for each $x \in X,$ there exists $\epsilon > 0$ such that $\mathbb{B}_\epsilon(x)\cap X$ is a closed set in $\mathbb{R}^n.$

\subsection{Set-valued maps}

For a set-valued map $F \colon \mathbb{R}^n \rightrightarrows \mathbb{R}^m$ we denote its {\em domain}, {\em range} (image), and {\em graph} by, respectively,
\begin{eqnarray*}
\mathrm{dom} F &:=&  \{x \in \mathbb{R}^n \ | \ F(x) \ne \emptyset\},\\
\mathrm{range} F &:=& \{y \in F(x) \ | \ x \in \mathrm{dom} F\},\\
\mathrm{graph} F &:=& \{(x, y) \in \mathbb{R}^n \times \mathbb{R}^m \ | \  x \in \mathrm{dom} F, y \in F(x)\}.
\end{eqnarray*}
The {\em inverse map} $F^{-1} \colon \mathbb{R}^m \rightrightarrows \mathbb{R}^n$ of the map $F$ is defined as
$$F^{-1}(y) := \{x \in \mathbb{R}^m \ | \ y \in F(x)\}.$$
It immediately follows from the above definitions that $\mathrm{dom} F^{-1} = \mathrm{range} F,$ $\mathrm{range} F^{-1} = \mathrm{dom} F,$ and that
\begin{eqnarray*}
\mathrm{graph} F^{-1} &:=& \{(y, x) \in \mathbb{R}^m \times \mathbb{R}^n \ | \  (x, y)  \in \mathrm{graph} F\}.
\end{eqnarray*}

Recall that the map $F$ is called {\em lower semicontinuous} at a point $x \in \mathrm{dom}F$ if for any $y \in F(x)$ and for any sequence of points $x^k \in \mathrm{dom}F$ converging to $x,$ there exists a sequence of points $y^k \in F(x^k)$ converging to $y.$ We say that $F$ is {\em lower semicontinuous} (on $\mathrm{dom} F$) if $F$ is lower semicontinuous at every point $x \in  \mathrm{dom}F.$ By definition, $F$ is lower semicontinuous at $x \in \mathrm{dom}F$ if and only if the inverse image of any open subset of $\mathbb{R}^m$ intersecting $F(x)$ is a neigbourhood of $x;$ confer~\cite[Proposition~1.4.4]{Aubin1990}. 

\begin{definition}
Let $F \colon \mathbb{R}^n \rightrightarrows \mathbb{R}^m$ be a set-valued map.

(i) $F$ is said to be an open map from $\mathrm{dom}F$ into $\mathrm{range} F$ if for every open set $U$ in $\mathrm{dom}F,$ the set $F(U)$ is open in $\mathrm{range} F;$

(ii) $F$ is said to be {\em H\"older metrically regular} if for each point $y^* \in \mathrm{range} F$ and for each compact set $K \subset \mathbb{R}^n,$ there exist constants $\epsilon > 0,$ $c > 0$ and $\alpha > 0$ such that
\begin{eqnarray*}
\mathrm{dist}(x, F^{-1}(y)) &\leq& 
c[\mathrm{dist}(y, F(x))]^\alpha
\end{eqnarray*}
for all $x \in K$ and all $y \in \mathbb{B}_\epsilon(y^*) \cap \mathrm{range}F.$

(iii) $F$ is said to be {\em H\"older metrically subregular} if for each point $y^* \in \mathrm{range} F$ and for each compact set $K \subset \mathbb{R}^n,$ there exist constants $c > 0$ and $\alpha > 0$ such that
\begin{eqnarray*}
\mathrm{dist}(x, F^{-1}(y^*)) &\leq& c[\mathrm{dist}(y^*, F(x))]^\alpha
\end{eqnarray*}
for all $x \in K.$

(iv) $F$ is said to be {\em lower pseudo-H\"older continuous} (on $\mathrm{dom}F$) if for each point $x^* \in \mathrm{dom} F$ and for each compact set $K \subset \mathbb{R}^m,$ there exist constants $\epsilon > 0,$ $c > 0$ and $\alpha > 0$ such that
\begin{eqnarray*}
F(x^*)  \cap K &\subset& F(x) + c \|x - x^*\|^\alpha \mathbb{B}
\end{eqnarray*}
for all $x \in \mathbb{B}_\epsilon(x^*)\cap\mathrm{dom}F.$

(v) $F$ is said to be {\em pseudo-H\"older continuous} (on $\mathrm{dom}F$) if for each point $x^* \in \mathrm{dom} F$ and for each compact set $K \subset \mathbb{R}^m,$ there exist constants $\epsilon > 0,$ $c > 0$ and $\alpha > 0$ such that
\begin{eqnarray*}
F(x^1)  \cap K &\subset& F(x^2) + c \|x^1 - x^2\|^\alpha \mathbb{B}
\end{eqnarray*}
for all $x^1, x^2 \in \mathbb{B}_\epsilon(x^*)\cap\mathrm{dom}F.$
\end{definition}

It should be noted that in the above definition we do not require that $x \in \mathrm{int} (\mathrm{dom} F)$ or $y \in \mathrm{int} (\mathrm{range} F).$

\subsection{Semialgebraic geometry}
Now, we recall some notions and results of semialgebraic geometry, which can be found in \cite{Benedetti1990, Bochnak1998, HaHV2017, Dries1996}.

\begin{definition}
A subset $S$ of $\mathbb{R}^n$ is called {\em semialgebraic}, if it is a finite union of sets of the form
$$\{x \in \mathbb{R}^n \ | \  f_i(x) = 0, \ i = 1, \ldots, k; f_i(x) > 0, \ i = k + 1, \ldots, p\},$$
where all $f_{i}$ are polynomials.
In other words, $S$ is a union of finitely many sets, each defined by finitely many polynomial equalities and inequalities.
A set-valued map $F \colon \mathbb{R}^n \rightrightarrows \mathbb{R}^m$ is said to be {\em semialgebraic}, if its graph is a semialgebraic set.
\end{definition}

\begin{example}
Recall that a polyhedral set is the intersection of a finite number of half-spaces and a set-valued map $F \colon \mathbb{R}^n \rightrightarrows \mathbb{R}^m$ is said to be {\em polyhedral} if its graph is a finite union of polyhedral sets. By definition, then polyhedral maps are semialgebraic.
\end{example}

A major fact concerning the class of semialgebraic sets is its stability under linear projections (see, for example, \cite{Bochnak1998}).

\begin{theorem}[Tarski--Seidenberg theorem] \label{TarskiSeidenbergTheorem}
The image of any semialgebraic set $S \subset \mathbb{R}^n$ under a projection to any linear subspace of $\mathbb{R}^n$ is a semialgebraic set.
\end{theorem}

\begin{remark}{\rm
As an immediate consequence of the Tarski--Seidenberg Theorem, we get semialgebraicity of any set $\{ x \in A \ | \  \exists y \in B,  (x, y) \in C \},$  provided that $A ,  B,$  and $C$  are semialgebraic sets in the corresponding spaces.
Also, $\{ x \in A \ | \  \forall y \in B,  (x, y) \in C \}$ is a semialgebraic set as its complement is the union of the complement of $A$  and the set $\{ x \in A \ | \  \exists y \in B,  (x, y) \not\in C \}.$
Thus, if we have a finite collection of semialgebraic sets, then any set obtained from them with the help of a finite chain of quantifiers is also semialgebraic.
In particular, if $F \colon \mathbb{R}^n \rightrightarrows \mathbb{R}^m$ is a semialgebraic set-valued map, then the inverse map $F^{-1} \colon \mathbb{R}^m \rightrightarrows \mathbb{R}^n$ is also semialgebraic and the sets $\mathrm{dom} F$ and $\mathrm{range} F$ are semialgebraic.
}\end{remark}

The following well-known lemmas will be of great importance for us.

\begin{lemma}[curve selection lemma] \label{LocalCurveSelectionLemma}
Let $A \subset \mathbb{R}^n$ be a semialgebraic set$,$ and let $a \in \mathrm{cl} A\setminus A.$
Then$,$ there exists a continuous semialgebraic curve $\phi \colon [0, \epsilon) \rightarrow \mathbb{R}^n$ such that $\phi(0) = a $ and $\phi(t) \in A$ for all $t \in (0, \epsilon).$
\end{lemma}

\begin{lemma}[growth dichotomy lemma] \label{GrowthDichotomyLemma}
Let $f \colon (0, \epsilon) \rightarrow {\mathbb R}$ be a semialgebraic function with $f(t) \ne 0$ for all $t \in (0, \epsilon).$ Then there exist constants $a \ne 0$ and $\alpha \in {\mathbb Q}$ such that $f(t) = at^{\alpha} + o(t^{\alpha})$ as $t \to 0^+.$
\end{lemma}

\begin{lemma}[monotonicity Lemma]\label{MonotonicityLemma}
Let $f\colon (a, b) \to \mathbb{R}$ be a semialgebraic function.
Then there are $a = a_0 < a_1 < \cdots< a_s < a_{s+1} = b$ such that, for each $i = 0,\ldots, s,$ the restriction $f|_{(a_i,a_{i+1})}$ is analytic, and either constant, or strictly increasing or strictly decreasing.
\end{lemma}

\begin{lemma}[uniform bounds on fibers]\label{Uniformbounds}
Let $A \subset \mathbb{R}^{m + n}$ be a semialgebraic set. Then there exists an integer number $N$ such that for all $x \in \mathbb{R}^m$ the set $\{y \in \mathbb{R}^n \ | \ (x, y) \in A\}$ has at most $N$ connected components.
\end{lemma}

The next theorem (see \cite{Bochnak1998, HaHV2017, Dries1996}) uses the concept of a cell whose definition we omit. We do not need the specific structure of cells described in the formal definition. For us, it will be sufficient to think of a $C^p$-cell of dimension $r$ as of an $r$-dimensional $C^p$-manifold, which is the image of the cube $(0, 1)^r$ under a semialgebraic $C^p$-diffeomorphism. As follows from the definition, an $n$-dimensional cell in $\mathbb{R}^n$ is an open set.

\begin{theorem}[cell decomposition theorem]
Let $A \subset \mathbb{R}^n$ be a  semialgebraic set. Then$,$ for any $p \in \mathbb{N},$ $A$ can be represented as a disjoint union of a finite number of cells of class~$C^p.$
\end{theorem}

By the cell decomposition theorem, for any $p \in \mathbb{N}$ and any nonempty semialgebraic subset $A$ of $\mathbb{R}^n,$ we can write $A$ as a disjoint union of finitely many semialgebraic $C^p$-manifolds of different dimensions. The {\em dimension} $\dim A$ of a nonempty semialgebraic set $A$ can thus be defined as the dimension of the manifold of highest dimension of its decomposition. This dimension is well defined and independent of the decomposition of $A.$ By convention, the dimension of the empty set is taken to be negative infinity. We will need the following result (see \cite{Bochnak1998, HaHV2017, Dries1996}).

\begin{lemma} \label{DimensionLemma}
Let $A \subset \mathbb{R}^n$ be a nonempty semialgebraic set. Then$,$ $\dim(\mathrm{cl} A\setminus A) < \dim A.$ In particular$,$ $\dim(\mathrm{cl} A) = \dim A.$
\end{lemma}

In the sequel we will make use of Hardt's semialgebraic triviality. We present a particular case--adapted to our needs--of a more general result: see \cite{Bochnak1998, Hardt1980, Dries1996} for the statement in its full generality.

\begin{theorem}[Hardt's semialgebraic triviality] \label{HardtTheorem}
Let $S$ be a semialgebraic set in $\mathbb{R}^n$ and $f \colon S \rightarrow\mathbb{R}$ a continuous semialgebraic map.
Then there are finitely many points $-\infty = t_0 < t_1 < \cdots < t_k = +\infty$ such that $f$ is semialgebraically trivial over each the interval $(t_i, t_{i + 1}),$ that is, there exists a semialgebraic set $F_i \subset \mathbb{R}^n$ and a semialgebraic homeomorphism
$h_i \colon f^{-1} (t_i, t_{i + 1}) \rightarrow (t_i, t_{i + 1}) \times F_i$ such that the composition $h_i$ with the projection $(t_i, t_{i + 1}) \times F_i \rightarrow (t_i, t_{i + 1}), (t, x) \mapsto t,$  is equal to the restriction of $f$ to $f^{-1} (t_i, t_{i + 1}).$
\end{theorem}

We also need the following lemma. 
\begin{lemma}\label{Lemma23}
Let $U$ and $V$ be open semialgebraic sets in $\mathbb{R}^n$ and let $f \colon U \to V$ be a continuous semialgebraic map.
If $f$ is open (i.e., $f$ maps open sets to open sets), then for all $y \in V$ the inverse image $f^{-1}(y)$ is finite.
\end{lemma}
\begin{proof}
See \cite[Theorem~3.10]{Denkowski2017} or \cite[Proposition, page~298]{Gamboa1996}).
\end{proof}

\section{Characterizations of the openness} \label{Section3}

We start with a result which provides a relation between the openness of (not necessarily semialgebraic) set-valued maps and the continuity of their inverse images;
see also \cite[Lemma~3.8]{Denkowski2017} and \cite[Theorem~2.3]{Rockafellar1985}.

\begin{lemma}\label{Lemma31}
Let $F\colon \mathbb{R}^n \rightrightarrows \mathbb{R}^m$ be a set-valued map with closed graph.
Then the following are equivalent$:$
\begin{enumerate}[{\rm (i)}]
  \item $F$ is an open map from $\mathrm{dom}F$ into $\mathrm{range} F;$
  \item the function
\begin{align*}
\mathbb{R}^n \times \mathrm{range} F \to \mathbb{R}, \ \ (x, y) \mapsto {\rm dist}(x, F^{-1}(y)),
\end{align*}
is continuous$.$
\end{enumerate}
\end{lemma}

\begin{proof}
(i) $\Rightarrow$ (ii).
Fix $a \in \mathbb{R}^n.$ It suffices to show that the function
\begin{align*}
\mathrm{range} F \to \mathbb{R}, \ \ y \mapsto {\rm dist}(a, F^{-1}(y)),
\end{align*}
is continuous. To this end, take any $y^* \in \mathrm{range} F$ and let $\{y^k\}$ be a sequence in $\mathrm{range} F$ such that $y^k \to y^*$ as $k \to \infty.$
We will prove
\begin{align*}
\lim_{k\to\infty}{\rm dist}(a, F^{-1}(y^k)) = {\rm dist}(a, F^{-1}(y^*)).
\end{align*}

Indeed, as the graph of $F$ is closed, the sets $F^{-1}(y^k)$ and $F^{-1}(y^*)$ are closed. It follows easily that there exist points $x^k \in F^{-1}(y^k)$ and $x^* \in F^{-1}(y^*)$ such that
\begin{align*}
\|a - x^k\|={\rm dist}(a, F^{-1}(y^k)) \quad \textrm{ and } \quad  \|a - x^*\|={\rm dist}(a, F^{-1}(y^*)).
\end{align*}
Take, arbitrarily, an open and bounded neighbourhood $U$ of $x^*$ in $\textrm{dom}F.$
By assumption, $F(U)$ is open in $\mathrm{range} F.$ Hence $y^k \in F(U)$ for all large $k.$
For all such $k,$ there exists $\bar x^k \in U$ such that $y^k \in F(\bar x^k),$ and so
\begin{align*}
\|a - x^k\| = {\rm dist}(a, F^{-1}(y^k)) \leq \|a - \bar x^k\|.
\end{align*}
Since the sequence $\{\bar x^k\}$ is contained in the bounded set $U,$ it follows that the sequences $\{x^k\}$ and $\{{\rm dist}(a, F^{-1}(y^k))\}$ are bounded.
Let $\{y^{k_l}\}$ be a subsequence of $\{y^k\}$ such that the limit $\lim_{l\to\infty} {\rm dist}(a, F^{-1}(y^{k_l}))$ exists.
It suffices to show that this limit equals to $\|a-x^*\|.$

Choosing subsequences if necessary, we may assume that the following limits exist:
\begin{align*}
x := \lim_{l\to\infty}x^{k_l} \ \textrm{ and } \ \bar x := \lim_{l\to\infty}\bar x^{k_l}.
\end{align*}
Clearly, $\bar x \in \mathrm{cl}U$-the closure of $U$ and, since the graph of $F$ is closed, it holds that
\begin{align*}
y^*\in F(x)\cap F(\bar x).
\end{align*}
Consequently,
\begin{eqnarray*}
\|a - x^*\| = {\rm dist}(a, F^{-1}(y^*)) &\leq& \|a - x\| = \lim_{l \to\infty}\|a - x^{k_l}\| \\ &\leq& \lim_{l\to\infty}\|a - \bar x^{k_l}\| = \|a - \bar x\|.
\end{eqnarray*}
Therefore,
\begin{align*}
\|a - x^*\| \leq \lim_{l\to\infty}{\rm dist}(a, F^{-1}(y^{k_l})) \leq \|a - \bar x\|.
\end{align*}
As $\bar x \in \mathrm{cl}U$ and $U$ is an {\em arbitrary} open neighbourhood of $x^*$ in $\mathrm{dom}F,$ we get easily
\begin{align*}
\|a - x^*\| = \lim_{l\to\infty}{\rm dist}(a, F^{-1}(y^{k_l})).
\end{align*}

(ii) $\Rightarrow$ (i).
Let $U$ be an open set in $\mathrm{dom}F$ and $y^* \in F(U),$ let $x^* \in U$ with $y^*\in F(x^*).$
We proceed by contradiction.
Assume that no neighbourhood of $y^*$ (relative to $\mathrm{range} F$) is contained in $F(U).$
Then there is a sequence $\{y^k\}$ in $\mathrm{range} F \setminus F(U)$ such that $y^k \to y^*$ as $k \to \infty.$
By assumption, then
\begin{align*}
\lim_{k\to\infty}{\rm dist}(x^*, F^{-1}(y^k)) = {\rm dist}(x^*, F^{-1}(y^*)) = 0.
\end{align*}
On the other hand, for each $k,$ the set $F^{-1}(y^k)$ is closed because the graph of $F$ is closed.
Consequently, there exists $x^k \in F^{-1}(y^k)$ such that
\begin{align*}
\|x^* - x^k\|={\rm dist}(x^*, F^{-1}(y^k)) .
\end{align*}
It follows that $x^k \to x^*$ as $k \to \infty$ and hence $x^k \in U$ for all large $k.$
For all such $k,$ $y^k \in F(x^k) \subset F(U),$ which is a contradiction.
\end{proof}

We also need the following result, which is a generalization of the \L ojasiewicz inequality (see, for example, \cite{Bochnak1998, HaHV2017}).

\begin{lemma}\label{Lemma32}
Let $K$ be a compact semialgebraic set in $\mathbb{R}^n.$ Let $\phi, \psi\colon K\to \mathbb{R}$ be non-negative$,$ semialgebraic functions satisfying the following conditions$:$
\begin{enumerate}[{\rm (i)}]
  \item $\phi$ is continuous$;$
  \item for any sequence $\{x^k\} \subset K$ converging to $\bar x \in K$ such that $\lim_{k \to \infty} \psi(x^k) = 0,$ it holds that $\psi(\bar{x}) = 0;$ 
  \item $\{x\in K \ | \  \psi(x) = 0\} \subset \{x\in K \ | \  \phi(x) = 0\}.$
  \end{enumerate}
  Then there exist constants $c>0$ and $\alpha>0$ such that
  $$c[\psi(x)]^\alpha \ge \phi(x) \quad \textrm{for all } \quad x \in K.$$
\end{lemma}

\begin{proof}
We may assume that the set $\psi^{-1}(0) = \{x\in K \ | \  \psi(x)=0\}$ is nonempty and different from $K.$ Otherwise, the lemma is trivial.

Let $M := \sup_{x \in K} \phi(x).$ By assumption, $M$ is finite and non-negative.
In light of Theorem~\ref{TarskiSeidenbergTheorem}, $\psi(K)$ is a semialgebraic set in $\mathbb{R},$ so it is a finite union of points and intervals. Observe that $0 \in \psi(K).$ There are two cases to consider.

\subsubsection*{Case 1:} $0$ is an isolated point of $\psi(K).$

Then there exists a real number $\epsilon > 0$ such that for all $x \in K,$ if $\psi(x) \le \epsilon$ then $\psi(x) = 0$ and so $\phi(x) = 0$ (by condition~(iii)).

On the other hand, for all $x \in K$ with $\psi(x) \ge \epsilon$ it holds that
$$\psi(x)\geq\epsilon=\frac{\epsilon}{(M + 1)} (M + 1) \geq \frac{\epsilon}{(M + 1)} \phi(x).$$
Therefore, for all $x\in K$ we have
$$\frac{(M + 1)}{\epsilon} \psi(x) \geq \phi(x),$$
which proves the lemma in this case.

\subsubsection*{Case 2:} $0$ is not an isolated point of $\psi(K).$

Then there exists a constant $T > 0$ such that $[0, T] \subset \psi(K)$ and so for all $t \in [0, T],$ the set $\psi^{-1}(t)$ is nonempty and bounded. This implies that the function
$$\mu\colon[0,T]\to\mathbb{R}, \quad t\mapsto \sup_{x\in\psi^{-1}(t)}\phi(x)$$
is well-defined and $\mu(0) = 0.$ Observe that, by Theorem~\ref{TarskiSeidenbergTheorem}, the function $\mu$ is semialgebraic.
We will show that $\mu$ is continuous at $t = 0.$ Indeed, if this is not the case, then there exists a real number $\delta>0$ and a sequence $\{t_k\}$ of real numbers with $t_k \to 0^+$ as $k \to \infty$ such that
$$\mu(t_k)\geq \delta \ \ \textrm{ for all } k.$$
By definition of the supremum, we can find $x^k\in\psi^{-1}(t_k)$ satisfying
$$\mu(t_k) \ge \phi(x^k)\geq\mu(t_k)-\frac{\delta}{2}\geq \delta -\frac{\delta}{2}
= \frac{\delta}{2} \ \ \textrm{ for all } k.$$
Note that $x^k$ belongs to the compact set $K.$
Passing a subsequence if necessary, we may assume that the sequence $\{x^k\}$ converges to some point $\bar x\in K.$ Then our assumptions imply that
\begin{eqnarray*}
  \phi(\bar x) &=& \lim_{k\to\infty} \phi(x^k)\geq\frac{\delta}{2}>0, \\
  \psi(\bar x) &=& 0 \quad \textrm{(because $\lim_{k \to \infty} \psi(x^k) = 0$}),
\end{eqnarray*}
which contradict the assumption that $\psi^{-1}(0)\subset \phi^{-1}(0).$ So, $\mu$ is continuous at $t=0.$

By Lemma~\ref{MonotonicityLemma}, there exists $\epsilon\in(0,T)$ such that the restriction of $\mu$ on $[0,\epsilon]$ is either constant or strictly monotone. There are two cases to consider.

\subsubsection*{Case 2.1:} $\mu$ is constant on $[0,\epsilon].$

Since $\mu(0)=0,$ $\mu(t)=0$ for all $t \in [0,\epsilon].$
It follows that for all $x\in K$ with $\psi(x)\leq \epsilon,$
$$\psi(x)\geq 0 = \phi(x).$$

On the other hand, by similar arguments as in Case~1, we can see that for all $x\in K$ with $\psi(x)\geq\epsilon,$
$$\psi(x) \geq \frac{\epsilon}{M + 1} \phi(x).$$
Therefore, for all $x\in K$ we have
$$\frac{M + 1}{\epsilon}\psi(x)\geq \phi(x).$$

\subsubsection*{Case 2.2:} $\mu$ is not constant on $[0,\epsilon].$

Thanks to Lemma~\ref{GrowthDichotomyLemma}, we can write
$$\mu(t)=at^{\alpha}+o(t^{\alpha}) \ \ \textrm{ as } t\to0^+$$
for some constants $a>0$ and $\alpha \in \mathbb{Q}.$ Observe that $\alpha > 0$ because $\mu(0) = 0$ and $\mu$ is continuous at $ t = 0.$ Letting $c_1:=2a$ and reducing $\epsilon$ (if necessary), we get
$$\mu(t)\leq c_1t^{\alpha} \ \ \textrm{ for } t\in[0,\epsilon].$$
Consequently, for all $x\in K$ with $\psi(x)\leq \epsilon,$ we have
$$\phi(x)\leq c_1[\psi(x)]^{\alpha}.$$

On the other hand, for all $x\in K$ with $\psi(x)\geq\epsilon,$ it is clear that
$$\psi(x)\geq\epsilon=\frac{\epsilon}{(M + 1)^{1/\alpha}} (M + 1)^{1/\alpha} \geq \frac{\epsilon}{(M + 1)^{1/\alpha}} [\phi(x)]^{1/\alpha},$$
or equivalently,
$$c_2[\psi(x)]^{\alpha} \geq \phi(x),$$
where $c_2 := \frac{M+1}{\epsilon^{\alpha}} > 0.$

Letting $c := \max\{c_1,c_2\} > 0,$ we get for all $x\in K,$
$$c[\psi(x)]^{\alpha} \geq \phi(x).$$
The lemma is proved.
\end{proof}

We are now in a position to state the main result of this section.

\begin{theorem}\label{Theorem31}
Let $F \colon \mathbb{R}^n \rightrightarrows \mathbb{R}^m$ be a semialgebraic set-valued map with closed graph.
Then the following are equivalent$:$
\begin{enumerate}[{\rm (i)}]
\item $F$ is an open map from $\mathrm{dom}F$ into $\mathrm{range} F$ and $\mathrm{range} F$ is locally closed;
\item $F$ is H\"older metrically regular;
\item $F^{-1}$ is pseudo-H\"older continuous;
\item $F^{-1}$ is lower pseudo-H\"older continuous.
\end{enumerate}
\end{theorem}
\begin{proof}
(i) $\Rightarrow$ (ii). Let $y^*\in\mathrm{range} F.$ By assumption, there exists $\epsilon>0$ such that       $\mathbb{B}_\epsilon(y^*)\cap\mathrm{range} F$ is a nonempty compact set in $\mathbb{R}^m.$
Let $K\subset \mathbb{R}^n$ be a nonempty compact set, and, for simplicity of notation, we write $V:=\mathbb{B}_\epsilon(y^*)\cap\mathrm{range} F.$ Then $K \times V$ is a nonempty compact set.

The function
$$\phi \colon K \times V \to \mathbb{R}, \quad (x, y) \mapsto \mathrm{dist}(x, F^{-1}(y)),$$
is well-defined and non-negative. Moreover, by Theorem~\ref{TarskiSeidenbergTheorem} and Lemma~\ref{Lemma31}, $\phi$ is a semialgebraic and continuous function.

Next, fix a constant $M > 0$ and define the nonnegative function $\psi\colon K \times V \to \mathbb{R}$ by 
$$\psi(x, y) := 
\begin{cases}
\mathrm{dist}(y, F(x)) & \textrm{ if } x \in \mathrm{dom} F, \\
M & \textrm{ otherwise.}
\end{cases}$$
In view of Theorem~\ref{TarskiSeidenbergTheorem}, $\psi$ is semialgebraic.
Moreover, since the graph of $F$ is closed, if $\psi(x, y) = 0$ then $y \in F(x)$ and so $\phi(x, y) = 0.$ Hence $\psi^{-1}(0)\subset\phi^{-1}(0).$

Now, let $\{(x^k, y^k)\} \subset K \times V$ be a sequence converging to $(\bar x, \bar y) \in K \times V$ such that $\lim_{k \to \infty} \psi(x^k, y^k) = 0.$ We will claim that $\psi(\bar{x}, \bar{y}) = 0.$ Indeed, without loss of generality, we may assume that $\psi(x^k, y^k) < M$ for all $k.$ Since the graph of $F$ is closed, there exists $z^k \in F(x^k)$ such that
$$\psi(x^k, y^k) = \mathrm{dist}(y^k, F(x^k)) = \|y^k - z^k\|.$$
On the other hand, the sequence $\{y^k\}$ converges to $\bar{y},$ and so it is bounded. It follows easily that the sequence $\{z^k\}$ is also bounded.
Passing a subsequence if necessary, we may assume that the sequence $\{z^k\}$ converges to some point $\bar z.$
Clearly $\bar z \in F(\bar x)$ because of the closedness of the graph of $F.$ Hence,
\begin{eqnarray*}
0 & = & \lim_{k\to\infty} \psi(x^k, y^k) \ = \ \lim_{k\to\infty}\mathrm{dist}(y^k, F(x^k)) \ =\ \lim_{k\to\infty} \|y^k - z^k\| \ = \ \|\bar y - \bar z\|,
\end{eqnarray*}
and so $\bar y = \bar z.$ Therefore, $\psi(\bar x, \bar y) = \psi(\bar x, \bar z) = \mathrm{dist}(\bar x, F^{-1}(\bar z)) = 0,$ as required.

Applying Lemma~\ref{Lemma32} to the functions $\phi$ and $\psi,$ we get constants $c > 0$ and $\alpha > 0$ such that for all $(x, y) \in K \times V,$
\begin{eqnarray*}
\phi(x, y) \le c [\psi(x, y)]^\alpha,
\end{eqnarray*}
or equivalently,
\begin{eqnarray*}
\mathrm{dist}(x, F^{-1}(y)) &\leq& c [\mathrm{dist}(y, F(x))]^{\alpha},
\end{eqnarray*}
which proves (ii).

(ii) $\Rightarrow$ (iii).
Let $y^*\in\mathrm{range} F$ and let $K\subset \mathbb{R}^n$ be a nonempty compact set. Then there exist constants $\epsilon > 0,$ $c > 0$ and $\alpha > 0$ such that
\begin{eqnarray*}
\mathrm{dist}(x, F^{-1}(y)) &\leq&  c[\mathrm{dist}(y, F(x))]^\alpha
\end{eqnarray*}
for all $x \in K$ and all $y \in V:= \mathbb{B}_\epsilon(y^*) \cap \mathrm{range}F.$ Fix $y^1, y^2 \in V,$ and take arbitrarily $x \in F^{-1}(y^1) \cap K.$ We have
\begin{eqnarray*}
\mathrm{dist}(x, F^{-1}(y^2)) &\leq& c [\mathrm{dist}(y^2, F(x))]^{\alpha} \ \le \ c \|y^1 - y^2\|^{\alpha}.
\end{eqnarray*}
Therefore
\begin{eqnarray*}
F^{-1}(y^1)  \cap K &\subset& F^{-1}(y^2) + c \|y^1 - y^2\|^\alpha \mathbb{B},
\end{eqnarray*}
which proves the desired claim.

(iii) $\Rightarrow$ (iv). This is obvious.

(iv) $\Rightarrow$ (i).
We proceed by contradiction. Assume that there exists an open set $U$ in $\textrm{dom} F$ such that $F(U)$ is not open in $\textrm{range}F,$ i.e., there exists a point $y^*\in F(U)$ and a sequence $\{y^k\}$ in $\textrm{range} F\backslash F(U)$ such that $y^k\to y^*$ as $k\to\infty.$
Let $x^*\in U$ with $y^*\in F(x^*).$
Condition~(iv) implies the existence of constants $c>0$ and $\alpha>0$ such that for sufficiently large $k,$
$$\textrm{dist}(x^*,F^{-1}(y^k))\leq c\|y^* - y^k\|^\alpha.$$
Since $F^{-1}(y^k)$ is closed, there exists $x^k\in F^{-1}(y^k)$ such that
\begin{eqnarray*}
\|x^*-x^k\|=\textrm{dist}(x^*,F^{-1}(y^k)).
\end{eqnarray*}
Therefore for sufficiently large $k,$
\begin{eqnarray*}
\|x^*-x^k\| \le c\|y^* - y^k\|^\alpha,
\end{eqnarray*}
and so $x^k\to x^*$ as $k\to\infty.$ Hence, $x^k\in U$ for sufficiently large $k.$ For all such $k,$ $y^k \in F(x^k) \subset F(U),$ which is a contradiction.
Therefore, $F$ is an open map from $\mathrm{dom}F$ into $\mathrm{range}F.$

Finally, take any $y^*\in \mathrm{range}F.$
Then $y^*\in F(x^*)$ for some $x^*\in\mathrm{dom}F.$
The assumption gives us constants $\epsilon>0,$ $c>0$ and $\alpha>0$ such that for all $y\in\mathbb{B}_\epsilon(y^*)\cap\mathrm{range}F,$
$$\mathrm{dist}(x^*,F^{-1}(y))\leq c\|y^* - y\|^{\alpha}.$$
We will show that the set $V:=\mathbb{B}_\epsilon(y^*)\cap\mathrm{range}F$ is closed in $\mathbb{R}^m.$
To see this, let $\{y^k\} \subset V$ be a sequence such that $y^k\to y$ as $k\to \infty.$
Clearly, $y\in\mathbb{B}_\epsilon(y^*).$
Moreover, for each $k,$ we can find $x^k\in F^{-1}(y^k)$ such that
$$\|x^*-x^k\| = \mathrm{dist}(x^*, F^{-1}(y^k)).$$
Consequently, for sufficiently large $k,$
$$\|x^*-x^k\| \leq c\|y^* - y\|^{\alpha} \leq c \epsilon^{\alpha}.$$
Hence the sequence $\{x^k\}$ is bounded, and so, it has a cluster point, say $x.$
Since the graph of $F$ is closed, $x\in F^{-1}(y)$ and so $y\in\mathrm{range}F.$
Therefore $\mathbb{B}_\epsilon(y^*)\cap \mathrm{range} F$ is a closed set.
\end{proof}

\begin{remark}\label{Remark1}
Let $F \colon \mathbb{R}^n \rightrightarrows \mathbb{R}^m$ be a polyhedral map.
It follows from the Tarski--Seidenberg theorem that the range of $F$ is a semialgebraic set. Hence, by a result of Kurdyka \cite{Kurdyka1992}, the range of $F$ is a finite union of sets $S_i,$ where each $S_i$ has the Whitney property with constant $M$: any two points $x, y \in S_i$ can be joined in $S_i$ by a piecewise smooth path of length $\le M\|x - y\|.$ From this, we can see that the main theorem in \cite{Gowda1996} still holds if we replace the assumption that the range of $F$ is convex by the assumption that the range of $F$ is connected. As we shall not use this ``improved'' statement, we leave the proof to the reader.
\end{remark}

With $G := F^{-1}$ Theorem~\ref{Theorem31} leads to the following corollary.
\begin{corollary}
Let $G \colon \mathbb{R}^m \rightrightarrows \mathbb{R}^n$ be a semialgebraic set-valued map with closed graph.
Then the following are equivalent$:$
\begin{enumerate}[{\rm (i)}]
\item $G$ is continuous $($i.e., the inverse image of an open set in $\mathrm{range} G$ under $G$ is open in $\mathrm{dom}G)$ and $\mathrm{dom}G$ is locally closed;
\item $G$ is pseudo-H\"older continuous.
\end{enumerate}
\end{corollary}

The following result shows that the H\"older metric subregularity holds for the class of semialgebraic set-valued maps with closed graph.
\begin{proposition}
Let $F \colon \mathbb{R}^n \rightrightarrows \mathbb{R}^m$ be a semialgebraic set-valued map with closed graph. Then $F$ is H\"older metrically subregular.
\end{proposition}
\begin{proof}
Indeed, fix a point $y^* \in \mathrm{range} F$ and a semialgebraic compact set $K \subset \mathbb{R}^n.$ The function
$$\phi \colon K \to \mathbb{R}, \quad x \mapsto \mathrm{dist}(x, F^{-1}(y^*)),$$
is non-negative and continuous. Moreover, by Theorem~\ref{TarskiSeidenbergTheorem}, $\phi$ is semialgebraic.

Next, fix a constant $M > 0$ and define the nonnegative function $\psi\colon K \to \mathbb{R}$ by 
$$\psi(x) := 
\begin{cases}
\mathrm{dist}(y^*, F(x)) & \textrm{ if } x \in \mathrm{dom} F, \\
M & \textrm{ otherwise.}
\end{cases}$$
In view of Theorem~\ref{TarskiSeidenbergTheorem}, $\psi$ is semialgebraic.
Moreover, since the graph of $F$ is closed, if $\psi(x) = 0$ then $y^* \in F(x)$ and so $\phi(x) = 0.$ Hence $\psi^{-1}(0)\subset\phi^{-1}(0).$

As in the proof of Theorem~\ref{Theorem31} we can see that for any sequence $\{x^k\} \subset K$ converging to $\bar x \in K$ such that $\lim_{k \to \infty} \psi(x^k) = 0,$ it holds that $\psi(\bar{x}) = 0.$ 

Finally, applying Lemma~\ref{Lemma32} to the semialgebraic functions $\phi$ and $\psi,$ we get the desired conclusion.
\end{proof}

\begin{corollary}\label{Corollary32}
Let $f\colon\mathbb{R}^n\to\mathbb{R}^m$ be a continuous semialgebraic map.
Then the following are equivalent$:$
\begin{enumerate}[{\rm (i)}]
\item $f$ is an open map from $\mathbb{R}^n$ into $\mathrm{range} f$ and $\mathrm{range} f$ is locally closed;
\item $f$ is H\"older metrically regular;
\item $f^{-1}$ is pseudo-H\"older continuous;
\item $f^{-1}$ is lower pseudo-H\"older continuous.
\end{enumerate}
\end{corollary}

\begin{proof}
Since the map $f$ is continuous, the graph of $f$ is closed.
Then the desired conclusion follow immediately from Theorem~\ref{Theorem31}.
\end{proof}

We end this section with the following result, which provides a characterization of the openness for continuous functions. Also note the work \cite[Theorem~1]{Puhl1998} 
(see also \cite[Proposition~5.11]{Durea2017}), for a convex continuous function, the equivalence between the openness at a linear rate and the non-minimality was shown.

\begin{proposition}
Let $f\colon\mathbb{R}^n \to \mathbb{R}$ be a continuous function. Then the following are equivalent$:$
\begin{enumerate}[{\rm (i)}]
\item $f$ is open;
\item $f$ has no extremum points.
\end{enumerate}
\end{proposition}

\begin{proof}
(i) $\Rightarrow$ (ii). This is obvious.

(ii) $\Rightarrow$ (i). Let $U$ be an open set in $\mathbb{R}^n$ and $y^* \in f(U),$ let $x^* \in U$ with $y^*= f(x^*).$ We proceed by contradiction.
Assume that no neighbourhood of $y^*$ (in $\mathbb{R}$) is contained in $f(U).$
Then there is a sequence $\{y^k\}$ in $\mathbb{R} \setminus f(U)$ converging to $y^*.$

Since $U$ is an open set containing $x^*,$ there exists a real number $r > 0$ such that $\mathbb{B}_r(x^*) \subset U.$ Since the closed ball $\mathbb{B}_r(x^*)$ is compact and connected set in $\mathbb{R}^n$ and since the function $f$ is continuous, the image $f(\mathbb{B}_r(x^*))$ is a compact and connected set in $\mathbb{R},$ and so it is a closed interval, say, $[a, b] \subset \mathbb{R}.$ Then 
$$a \le y^* = f(x^*) \le b.$$ 
If $a = y^*$ (resp., $b = y^*$) then $f(x) \ge y^*$ (resp., $f(x) \le y^*)$ for all $x \in \mathbb{B}_r(x^*),$ and so $x^*$ is a local minimizer (resp., maximizer) of $f,$ which contradicts our assumption. Therefore $a < y^* < b.$ It follows that $a < y^k < b$ for all large $k.$ For all such $k,$ $y^k \in f(\mathbb{B}_r(x^*)) \subset f(U),$ which is a contradiction.
\end{proof}

\section{Application to the variational inequality} \label{Section4}

In this section, we apply our previous analysis to semialgebraic variational inequalities.
So consider the variational inequality problem formulated in the introduction section:
\begin{equation} \label{VI}
{\rm Find} \  x \in C \ \textrm{ such that } \ \langle p + f(x), x' - x \rangle \ \ge \  0 \quad \textrm{ for all } \quad x' \in C, \tag{VI}
\end{equation}
where $p \in \mathbb{R}^n$ is a parameter vector, $f \colon \mathbb{R}^n \to \mathbb{R}^n$ is a continuous semialgebraic map, and $C \subset \mathbb{R}^n$ is a closed convex semialgebraic set. 

Let $\mathscr{S} \colon \mathbb{R}^n \rightrightarrows  \mathbb{R}^n, p \mapsto \mathscr{S}(p),$ be the solution map associated to the variational inequality~\eqref{VI} and let $\mathscr{F}\colon\mathbb{R}^n\to\mathbb{R}^n,$ $u\mapsto \mathscr{F}(u),$ be the {\em normal map} defined by
\begin{eqnarray*}
\mathscr{F}(u) &:=& f(\Pi_C(u)) + u - \Pi_C(u),
\end{eqnarray*}
where $\Pi_C$ is the Euclidean projection onto the set $C.$

\begin{lemma}\label{Lemma41}
The following relations hold:
\begin{eqnarray*}
\mathscr{S}(p) = \Pi_C\big (\mathscr{F}^{-1}(-p) \big) \quad \textrm{ and } \quad
\mathscr{F}^{-1}(-p) = \{x - f(x) - p \ | \ x \in \mathscr{S}(p)\}.
\end{eqnarray*}
\end{lemma}
\begin{proof}
See, for example, \cite[Proposition~1.5.9]{Facchinei2003}.
\end{proof}

The main result of this section is as follows.

\begin{theorem}\label{Theorem41}
With the above notation, the following statements are equivalent:
\begin{enumerate}[{\rm (i)}]
\item $\mathscr{S}$ is lower semicontinuous and $\mathrm{dom}\mathscr{S}$ is locally closed;
\item $\mathscr{S}$ is pseudo-H\"older continuous;
\item $\mathscr{S}$ is lower pseudo-H\"older continuous;
\item $\mathscr{S}^{-1}$ is an open map from $\mathrm{range}\mathscr{S}$ into $\mathrm{dom}\mathscr{S}$  and $\mathrm{dom}\mathscr{S}$ is locally closed;
\item $\mathscr{S}^{-1}$ is H\"older metrically regular;
\item $\mathscr{F}$ is an open map from $\mathbb{R}^n$ into $\mathrm{range} \mathscr{F}$ and $\mathrm{range}  \mathscr{F}$ is locally closed;
\item $\mathscr{F}$ is H\"older metrically regular;
\item $\mathscr{F}^{-1}$ is pseudo-H\"older continuous;
\item $\mathscr{F}^{-1}$ is lower pseudo-H\"older continuous.
\end{enumerate}
\end{theorem}
\begin{proof}
Since $f$ is continuous, it is easy to check that the graph of $\mathscr{S}$ is closed and the function $\mathscr{F}$ is continuous. Furthermore, in view of Theorem~\ref{TarskiSeidenbergTheorem}, the maps $\mathscr{S}$ and $\mathscr{F}$ are semialgebraic. By Theorem~\ref{Theorem31} and Corollary~\ref{Corollary32}, it suffices to show the implications (i) $\Rightarrow$ (vi), (ix) $\Rightarrow$ (iii) and (iii) $\Rightarrow$ (i).

(i) $\Rightarrow$ (vi).
Let $U$ be an open set in $\mathbb{R}^n$ and we show that $\mathscr{F}(U)$ is open in $\mathrm{range}\mathscr{F}.$
Suppose on the contrary that there exists a point $-p\in \mathscr{F}(U)$ and a sequence $\{-p^k\}$ in $\mathrm{range}\mathscr{F}$ such that $-p^k\to-p$ and $-p^k\notin \mathscr{F}(U)$ for sufficiently large $k.$
Let $u\in U$ with $-p=\mathscr{F}(u).$ By Lemma~\ref{Lemma41}, $x := \Pi_C(u)\in\Pi_C(\mathscr{F}^{-1}(-p))=\mathscr{S}(p).$ Hence
$$-p=\mathscr{F}(u)=f(\Pi_C(u)) + u - \Pi_C(u)=f(x)+u-x,$$
and so, $u=x-p-f(x).$

On the other hand, since $\mathscr{S}$ is lower semicontinuous, there exists a sequence $\{x^k\}\subset \mathscr{S}(p^k)$ such that $x^k\to x.$
It follows from Lemma~\ref{Lemma41} that there exists $u^k\in \mathscr{F}^{-1}(-p^k)$ such that $x^k=\Pi_C(u^k),$ and so,
$$-p^k =\mathscr{F}(u^k)  = f(\Pi_C(u^k)) + u^k - \Pi_C(u^k) =f(x^k)+u^k-x^k.$$
Consequently,
$$\lim_{k \to \infty} u^k = \lim_{k \to \infty} (x^k-p^k-f(x^k)) = x-p-f(x)=u.$$
Since $U$ is open containing $u,$ $u^k\in U$ for all large $k.$
For all such $k,$ $-p^k = \mathscr{F}(u^k) \in \mathscr{F}(U),$ which is a contradiction; thus $\mathscr{F}$ is an open map from $\mathbb{R}^n$ into $\mathrm{range}\mathscr{F}.$

Observe that $\mathrm{dom}\mathscr{S} = -\mathrm{range}\mathscr{F}$ because of Lemma~\ref{Lemma41}. This, together with the assumption, implies that the set $\mathrm{range}\mathscr{F}$ is locally closed.

(ix) $\Rightarrow$ (iii).
Let $p^*\in \mathrm{dom}\mathscr{S}$ and $K$ be a compact set in $\mathbb{R}^n$ such that $\mathscr{S}(p^*)\cap K\neq\emptyset.$ Since $f$ is continuous and $K$ is compact, we can find a real number $R > 0$ satisfying
$$\{x -  f(x) - p^* \ | \ x \in \mathscr{S}(p^*) \cap K\} \subset \mathbb{B}_R,$$
which, together with Lemma~\ref{Lemma41}, yields
$$\mathscr{S}(p^*)\cap K\subset\Pi_C(\mathscr{F}^{-1}(-p^*) \cap \mathbb{B}_R).$$
Since $\mathscr{F}^{-1}$ is lower pseudo-H\"older continuous, there exist constants $\epsilon>0,$ $c>0$ and $\alpha>0$ such that
$$\mathscr{F}^{-1}(-p^*)\cap \mathbb{B}_R\subset \mathscr{F}^{-1}(-p)+c\|p-p^*\|^{\alpha}\mathbb{B}$$
for all $p\in\mathbb{B}_\epsilon(p^*)\cap(-\mathrm{range}\mathscr{F})=\mathbb{B}_\epsilon(p^*)\cap\mathrm{dom}\mathscr{S}.$
It suffices to show that
$$\mathscr{S}(p^*)\cap K\subset \mathscr{S}(p)+c\|p-p^*\|^\alpha\mathbb{B}$$
for all $p\in \mathbb{B}_\epsilon(p^*)\cap \mathrm{dom}\mathscr{S}.$
To see this, take any $x \in \mathscr{S}(p^*)\cap K$ and $p\in \mathbb{B}_\epsilon(p^*)\cap\mathrm{dom}\mathscr{S}.$
Then there exists $u \in \mathscr{F}^{-1}(-p^*)\cap\mathbb{B}_R$ such that $x = \Pi_C(u).$
Note that $\mathscr{F}$ is continuous on $\mathbb{R}^n$.
So $\mathscr{F}^{-1}(-p)$ is a closed set in $\mathbb{R}^n,$ and hence, there exists $v \in \mathscr{F}^{-1}(-p)$ such that
$$\|u - v\| = \mathrm{dist}(u, \mathscr{F}^{-1}(-p)).$$
Letting $y := \Pi_C(v) \in \Pi_C(\mathscr{F}^{-1}(-p)) = \mathscr{S}(p),$ we get
\begin{equation*}
\mathrm{dist}(x, \mathscr{S}(p)) \leq \|x - y\| = \|\Pi_C(u) - \Pi_C(v)\|.
\end{equation*}
Note that the projection $\Pi_C\colon \mathbb{R}^n\to C$ is nonexpansive. Consequently,
$$\mathrm{dist}(x, \mathscr{S}(p)) \leq \|u - v\| = \mathrm{dist}(u, \mathscr{F}^{-1}(-p))\leq c\|p - p^*\|^\alpha.$$
Therefore $\mathscr{S}$ is lower pseudo-H\"older continuous.

(iii) $\Rightarrow$ (i).
Assume that $\mathscr{S}$ is lower pseudo-H\"older continuous. Then, analysis similar to that in the proof of Theorem~\ref{Theorem31} shows that $\mathrm{dom}\mathscr{S}$ is locally closed. So it remains to prove that $\mathscr{S}$ is lower semicontinuous. To this end, take any $p^*\in\mathrm{dom}\mathscr{S}$ and $x^*\in\mathscr{S}(p^*).$
Let $\{p^k\}$ be a sequence in $\mathrm{dom}\mathscr{S}$ such that $p^k\to p^*$ as $k\to \infty.$
There exist constants $\epsilon>0,$ $c>0$ and $\alpha>0$ such that
$$\mathrm{dist}(x^*, \mathscr{S}(p))\leq c\|p-p^*\|^\alpha$$
for all $p\in\mathbb{B}_\epsilon(p^*)\cap\mathrm{dom}\mathscr{S}.$
So, for sufficiently large $k,$
$$\mathrm{dist}(x^*,\mathscr{S}(p^k))\leq c\|p^k-p^*\|^\alpha.$$
Since the set $\mathscr{S}(p^k)$ is closed, there exists $x^k\in\mathscr{S}(p^k)$ such that
$$\|x^* - x^k\|\leq c\|p^k-p^*\|^\alpha.$$
Letting $k\to \infty,$ we get $x^k \to x^*.$ Therefore, $\mathscr{S}$ is lower semicontinuous.
\end{proof}

The next result, together with Theorem~\ref{Theorem41}, shows that the lower semicontinuity of the map $\mathscr{S}$ ensures the finiteness and the pseudo-H\"older continuity of the map $\mathscr{S}.$ It should be mentioned at this point that we make no assumption about how the set $C$ is presented; compare \cite[Theorem~4.3]{Lee2018}.

\begin{theorem}
If the map $\mathscr{S}$ is lower semicontinuous then there exists a natural number $N$ such that for all $p \in \mathrm{cl}(\mathrm{int}(\mathrm{dom}\mathscr{S}))$ the set $\mathscr{S}(p)$ has at most $N$ points.
\end{theorem}
\begin{proof}
Assume that the map $\mathscr{S}$ is lower semicontinuous. By Lemma~\ref{Uniformbounds}, it suffices to show that for all $p \in \mathrm{cl}(\mathrm{int}(\mathrm{dom}\mathscr{S}))$ the set $\mathscr{S}(p)$ is finite.

Tracing the argument in the proof of Theorem~\ref{Theorem41}, we can see that $\mathscr{F}$ is an open map from $\mathbb{R}^n$ into $\mathrm{range} \mathscr{F}.$ The set $V := \mathrm{int}(\mathrm{range}\mathscr{F})$ is an open set in $\mathbb{R}^n$ so is $U := \mathscr{F}^{-1}(V)$ (because the map $\mathscr{F}$ is continuous). Hence, the restriction of $\mathscr{F}$ on $U$ is an open map from $U$ into $V.$ This, combined with Lemma~\ref{Lemma23}, implies that for all $p \in V,$ the inverse image $\mathscr{F}^{-1}(p)$ is finite. Hence, by Lemma~\ref{Lemma41}, for all $p \in W := \mathrm{int}({\mathrm{dom}} \mathscr{S})$ the set $\mathscr{S}(p)$ is finite.

Take any $p^* \in \mathrm{cl}{W} \setminus W.$ If $p^* \not \in {\mathrm{dom}} \mathscr{S},$ then $\mathscr{S}(p^*) = \emptyset$ and there is nothing to prove. So assume that $p^* \in {\mathrm{dom}} \mathscr{S}.$ The remaining part of the proof is similar to that of \cite[Theorem~4.3]{Lee2018}, and we give it here 
for the convenience of the reader. By Lemma~\ref{LocalCurveSelectionLemma}, there exists a continuous semialgebraic curve $p \colon [0, \epsilon) \rightarrow \mathbb{R}^n$ such that $p(0) = p^*$ and $p(t) \in W$ for all $t \in (0,\epsilon).$
Let $A := \{(t, x) \in (0, \epsilon) \times C \ | \  x \in \mathscr{S}(p(t)) \}$ and consider the continuous  map  $\pi \colon A \rightarrow (0,\epsilon),  (t, x) \mapsto t.$
It is easy to check that $A$ is a semialgebraic set, and so, $\pi$ is a semialgebraic map. Applying Theorem~\ref{HardtTheorem} to $\pi$ (and shrinking $\epsilon > 0$ if necessary), we can see that the set $A$ is homeomorphic to $(0, \epsilon) \times \pi^{-1}(t_0)$ for some $t_0 \in (0, \epsilon).$  Note that $\dim \pi^{-1}(t_0) = 0$ since $\pi^{-1}(t_0) = \mathscr{S}(p(t_0))$ is a finite set.
Therefore,
\begin{eqnarray*}
  \dim A
&=& \dim ((0, \epsilon) \times\pi^{-1}(t_0)) \\
&=& \dim (0, \epsilon) + \dim\pi^{-1}(t_0) \\
&=& \dim (0, \epsilon)\\
&=& 1.
\end{eqnarray*}
Thanks to Lemma~\ref{DimensionLemma},
$\dim(\mathrm{cl} A \setminus A) < \dim A = 1,$ and so, $\mathrm{cl} A \setminus A$ is a finite set. On the other hand, since the map $\mathscr{S}$ is lower semicontinuous at $p^*,$ we have that
$$\mathscr{S}(p^*) \subset \{x : (0, x) \in \mathrm{cl} A \setminus A\}.$$
Consequently, $\mathscr{S}(p^*) $ is a finite set. The theorem follows.
\end{proof}

\section{Final remarks} \label{Section5}

The following example shows that, in general, we can not replace the word {\em H\"older} by {\em Lipschitz.}
\begin{example}{\rm
Let $f \colon \mathbb{R} \to \mathbb{R}$ be the polynomial function defined by $f(x) := x^d$ for some odd integer $d > 1.$ Then $f$ is strictly increasing and so it is an open map. Clearly, the inverse function $f^{-1} \colon \mathbb{R} \to \mathbb{R}, y \mapsto y^{\frac{1}{d}},$ is pseudo-H\"older continuous at $0$ with the exponent $\alpha = \frac{1}{d} < 1,$ but $f^{-1}$ is not Lipschitz continuous at $0.$
}\end{example}

The following example shows that the H\"older continuity properties may not hold on unbounded sets.

\begin{example}
Consider the semialgebraic function $f \colon \mathbb{R} \to \mathbb{R}$ defined by
$$f(x) :=
\begin{cases}
\frac{x^2}{x^2 + 1} & \textrm{ if } x \ge 0, \\
-\frac{x^2}{x^2 + 1} & \textrm{ otherwise.}
\end{cases}$$
A direct computation shows that $f$ is continuously differentiable and strictly increasing. So $f$ is an open map. By Corollary~\ref{Corollary32}, $f^{-1}$ is lower pseudo-H\"older continuous on $\mathrm{dom}f  = (-1, 1),$ i.e.,
for each compact set $K \subset \mathbb{R}$ and each $y^* \in (-1, 1)$ there exist constants $\epsilon > 0,$ $c > 0$ and $\alpha > 0$ such that
\begin{eqnarray*}
f^{-1}(y^*)  \cap K &\subset& f^{-1}(y) + c |y - y^*|^\alpha \mathbb{B} \quad \textrm{ for all } \quad y \in \mathbb{B}_\epsilon(y^*).
\end{eqnarray*}
In general the inclusion is no longer true if $K$ is a unbounded set. Indeed, let $K = \mathbb{R},$ $y^* = 0$ and $y^k := 1 - \frac{1}{k}$ for $k \ge 1.$ A simple calculation shows that $f^{-1}(y^*) = \{0\}$ and 
$f^{-1}(y^k) = \{\sqrt{k - 1}\}.$ Hence
$$\lim_{k \to \infty} \mathrm{dist}(0, f^{-1}(y^k)) = 
\lim_{k \to \infty} \sqrt{k - 1} = +\infty.$$
Therefore, there do not exist constants $\epsilon > 0,$ $c > 0$ and $\alpha > 0$ such that
\begin{eqnarray*}
f^{-1}(y^*) &\subset& f^{-1}(y) + c |y - y^*|^\alpha \mathbb{B}  \quad \textrm{ for all } \quad y \in \mathbb{B}_\epsilon(y^*).
\end{eqnarray*}
\end{example}

Next, we give a negative answer to the question of Dontchev--Rockafellar mentioned in the introduction.
\begin{example} \label{Example53}
We identify $\mathbb{C}$ with $\mathbb{R}^2$ and consider the polynomial map $f \colon \mathbb{C} \to \mathbb{C}, x \mapsto x^2$ (that is $f(x_1, x_2) :=
(x_1^2 - x_2^2, 2 x_1 x_2)$ for $(x_1, x_2) \in \mathbb{R}^2).$ Clearly,
$$\#f^{-1}(y) =
\begin{cases}
1 & \textrm{ if } y = 0, \\
2 &\textrm{ otherwise.}
\end{cases}$$
Moreover, the Jacobian of $f$ is nonnegative. In light of \cite[Theorem]{Gamboa1996} (see also \cite[Theorem~3.14]{Denkowski2017}), $f$ is an open map. Note that $f$ is surjective but not injective (compare \cite[Theorem~5]{Gowda1996}). A simple calculation shows that the inverse map $f^{-1} \colon \mathbb{C} \rightrightarrows \mathbb{C}$ is not Lipschitz continuous around $y = 0.$ However, for all $y$ and $y'$ near $0 \in \mathbb{C},$ we have
$$f^{-1}(y) \subset f^{-1}(y') + c|y - y'|^{\frac{1}{2}}\mathbb{B}$$
for some $c > 0.$ We now let $C = \mathbb{C}.$ Then the solution map $\mathscr{S}$ associated to the problem~\eqref{VI} is given by
$\mathscr{S}(p) = f^{-1}(-p)$ for $p \in \mathbb{C}.$ Therefore, the map $\mathscr{S}$ is lower semicontinuous on $\mathbb{C}$ but it
is not single-valued and also not Lipschitz continuous around $0 \in \mathbb{C}.$
\end{example}

Finally, we note that an assumption like semialgebraicity (or, more generally, tameness) is crucial for results presented in this paper. To see this, consider the following example; note that there is an inaccuracy in \cite[Example~4.5]{Borwein1988}.

\begin{example} \label{Example5}
Let $f \colon \mathbb{R} \to \mathbb{R}$ be the function defined by
$$f(x) :=
\begin{cases}
e^{-1/x} & \textrm{ if } x > 0, \\
0 & \textrm{ if } x=0, \\
-e^{1/x} &\textrm{ if } x<0.
\end{cases}$$
A direct computation shows that $f$ is a homeomorphism from $\mathbb{R}$ into $\mathrm{range} f = (-1, 1).$ So, $f$ is an open map. On the other hand, $f$ is not lower pseudo-H\"older continuous on $\mathrm{range} f.$
Indeed, by contradiction,  for $y^*=0$ and $K \subset \mathbb{R}$ with $f^{-1}(y^*)\cap K \ne \emptyset,$ assume that there exist constants $\epsilon>0,$ $c>0$ and $\alpha>0$ such that
$$f^{-1}(y^*)\cap K\subset f^{-1}(y)+c|y-y^*|^{\alpha}$$
for all $y \in \mathbb{B}_\epsilon(y^*).$
Then, this inclusion can be rewritten as
$$\mathrm{dist}(0, f^{-1}(y))\leq c|y|^{\alpha}.$$
Since the function $f$ is strictly increasing, it holds that
$$|x|\leq c|f(x)|^{\alpha}$$
for all $x$ near $0.$
This implies that
$$0<\frac{1}{c}\leq\lim_{x\to0}\frac{|f(x)|^{\alpha}}{|x|}=0,$$
which is impossible.
\end{example}


\begin{thebibliography}{10}

\bibitem{Aubin1984}
J.~P. Aubin.
\newblock Lipschitz behavior of solutions to convex minimization problems.
\newblock {\em Math. Oper. Res.}, 9(1):87--111, 1984.

\bibitem{Aubin1990}
J.-P. Aubin and H.~Frankowska.
\newblock {\em Set-Valued Analysis}.
\newblock Birkh\"auser, Boston, MA, 1990.

\bibitem{Benedetti1990}
R.~Benedetti and J.-J. Risler.
\newblock {\em Real Algebraic and Semi-Algebraic Sets}.
\newblock Actualit\'es Math\'ematiques. Hermann, Paris, 1990.

\bibitem{Bochnak1998}
J.~Bochnak, M.~Coste, and M.-F. Roy.
\newblock {\em Real Algebraic Geometry}, volume~36.
\newblock Springer, Berlin, 1998.

\bibitem{Bonnas2000}
J.~F. Bonnans and A.~Shapiro.
\newblock {\em Perturbation Analysis of Optimization Problems}.
\newblock Springer, New York, 2000.

\bibitem{Borwein1988}
J.~M. Borwein and D.~M. Zhuang.
\newblock Verifiable necessary and sufficient conditions for regularity of
  set-valued and single-valued maps.
\newblock {\em J. Math. Anal. Appl.}, 134:441--459, 1988.

\bibitem{Chuong2016}
T.~D. Chuong and D.~S. Kim.
\newblock H\"older-like property and metric regularity of a positive-order for
  implicit multifunctions.
\newblock {\em Math. Oper. Res.}, 41(2):596--611, 2016.

\bibitem{Denkowski2017}
M.~Denkowski and J.-J. Loeb.
\newblock On open analytic and subanalytic mappings.
\newblock {\em Complex Var. Elliptic Equ.}, 62(1):27--46, 2017.

\bibitem{Dontchev1996}
A.~L. Dontchev and R.~T. Rockafellar.
\newblock Characterizations of strong regularity for variational inequalities
  over polyhedral convex sets.
\newblock {\em SIAM J. Optim.}, 4(4):1087--1105, 1996.

\bibitem{Dontchev2009}
A.~L. Dontchev and R.~T. Rockafellar.
\newblock {\em Implicit {{F}}unctions and {{S}}olution {{M}}appings. {{A}}
  {{V}}iew from {{V}}ariational {{A}}nalysis}.
\newblock Springer Monogr. Math. Springer, Dordrecht, 2009.

\bibitem{Durea2017}
M.~Durea, M.~Pan\c{t}iruc, and R.~Strugariu.
\newblock A new type of directional regularity for mappings and applications to
  optimization.
\newblock {\em SIAM J. Optim.}, 27(2):1204--1229, 2017.

\bibitem{Facchinei2003}
F.~Facchinei and J.~S. Pang.
\newblock {\em Finite-Dimensional Variational Inequalities and Complementarity
  Problem, vols I, II}.
\newblock Springer, New-York, 2003.

\bibitem{Frankowska1990}
H.~Frankowska.
\newblock Some inverse mapping theorems.
\newblock {\em Ann. Inst. H. Poincaré Anal. Non Lin\'eaire}, 7(3):183--234,
  1990.

\bibitem{Frankowska2012}
H.~Frankowska and M.~Quincampoix.
\newblock H\"older metric regularity of set-valued maps.
\newblock {\em Math. Program. Ser. A}, 132(1--2):333--354, 2012.

\bibitem{Gamboa1996}
M.~Gamboa and F.~Ronga.
\newblock On open real polynomial maps.
\newblock {\em J. Pure Appl. Algebra}, 110(3), 1996.

\bibitem{Gfrerer2016}
H.~Gfrerer and D.~Klatte.
\newblock Lipschitz and {{H}}\"older stability of optimization problems and
  generalized equations.
\newblock {\em Math. Program. Ser. A}, 158(1--2):35--75, 2016.

\bibitem{Gowda1996}
M.~S. Gowda and R.~Sznajder.
\newblock On the {{L}}ipschitzian properties of polyhedral multifunctions.
\newblock {\em Math. Program. Ser. A}, 74(3):267--278, 1996.

\bibitem{HaHV2017}
H.~V. H\`a and T.~S. Ph\d{a}m.
\newblock {\em Genericity in Polynomial Optimization}, volume~3 of {\em Series
  on Optimization and Its Applications}.
\newblock World Scientific, Singapore, 2017.

\bibitem{Hardt1980}
R.~M. Hardt.
\newblock Semi-algebraic local-triviality in semi-algebraic mappings.
\newblock {\em Amer. J. Math.}, 102(2):291--302, 1980.

\bibitem{Hirsch2002}
M.~W. Hirsch.
\newblock Jacobians and branch points of real analytic open maps.
\newblock {\em Aequationes Math.}, 63(1--2):76--80, 2002.

\bibitem{Huynh2008}
V.~N. Huynh and M.~Th\'era.
\newblock Error bounds in metric spaces and application to the perturbation
  stability of metric regularity.
\newblock {\em SIAM J. Optim.}, 19(1):1--20, 2008.

\bibitem{Ioffe2013}
A.~Ioffe.
\newblock Nonlinear regularity models.
\newblock {\em Math. Program. Ser. B}, 139(1--2):223--242, 2013.

\bibitem{Ioffe2017}
A.~Ioffe.
\newblock {\em Variational Analysis of Regular Mappings. Theory and
  Applications}.
\newblock Springer Monographs in Mathematics. Springer, 2017.

\bibitem{Ioffe2018}
A.~D. Ioffe.
\newblock On variational inequalities over polyhedral sets.
\newblock {\em Math. Program. Ser. B}, 168(1--2):261--278, 2018.

\bibitem{Klatte2002}
D.~Klatte and B.~Kummer.
\newblock {\em Nonsmooth Equations in Optimization. Regularity, Calculus,
  Methods and Applications}, volume~60 of {\em Nonconvex Optim. Appl.}
\newblock Kluwer Academic Publishers, Dordrecht, 2002.

\bibitem{Kruger2015}
A.~Kruger.
\newblock Error bounds and {{H}}\"older metric subregularity.
\newblock {\em Set-Valued Var. Anal.}, 23(4):705--736, 2015.

\bibitem{Kurdyka1992}
K.~Kurdyka.
\newblock On a subanalytic stratification satisfying a {{W}}hitney property
  with exponent 1.
\newblock In {\em Real algebraic geometry (Rennes 1991)}, volume 1524, pages
  316--322, Berlin, 1992. Springer.

\bibitem{Lee2018}
J.~H. Lee, G.~M. Lee, and T.~S. Ph\d{a}m.
\newblock Genericity and {{H}}\"older stability in semi-algebraic variational
  inequalities.
\newblock {\em J. Optim. Theory Appl.}, 178(1):56--77, 2018.

\bibitem{Li2012}
G.~Li and B.~S. Mordukhovich.
\newblock H\"older metric subregularity with applications to proximal point
  method.
\newblock {\em SIAM J. Optim.}, 22(4):1655--1684, 2012.

\bibitem{Mordukhovich2006}
B.~S. Mordukhovich.
\newblock {\em Variational Analysis and Generalized Differentiation, I: Basic
  Theory; II: Applications}.
\newblock Springer, Berlin, 2006.

\bibitem{Mordukhovich2018}
B.~S. Mordukhovich.
\newblock {\em Variational Analysis and Applications}.
\newblock Springer, New York, 2018.

\bibitem{Mordukhovich2015}
B.~S. Mordukhovich and W.~Ouyang.
\newblock Higher-order metric subregularity and its applications.
\newblock {\em J. Global Optim.}, 63(4):777--795, 2015.

\bibitem{Penot1989}
J.-P. Penot.
\newblock Metric regularity, openness and {{L}}ipschitz behavior of maps.
\newblock {\em Nonlinear Anal.}, 13:629--643, 1989.

\bibitem{Puhl1998}
H.~P\"uhl.
\newblock Convexity and openness with linear rate.
\newblock {\em J. Math. Anal. Appl.}, 227:382--395, 1998.

\bibitem{Robinson1980}
S.~M. Robinson.
\newblock Strongly regular generalized equations.
\newblock {\em Math. Oper. Res.}, 5(1):43--62, 1980.

\bibitem{Rockafellar1985}
R.~T. Rockafellar.
\newblock Lipschitzian properties of multifunctions.
\newblock {\em Nonlinear Anal.}, 9(8):867--885, 1985.

\bibitem{Rockafellar1998}
R.~T. Rockafellar and R.~Wets.
\newblock {\em Variational Analysis}, volume 317 of {\em Grundlehren Math.
  Wiss.}
\newblock Springer-Verlag, Berlin, 1998.

\bibitem{Schirotzek2007}
W.~Schirotzek.
\newblock {\em Nonsmooth Analysis}.
\newblock Universitext. Springer, Berlin, 2007.

\bibitem{Scholtes2012}
S.~Scholtes.
\newblock {\em Introduction to Piecewise Differentiable Equations}.
\newblock SpringerBriefs in Optimization. Springer, New York, 2012.

\bibitem{Dries1996}
L.~van~den Dries and C.~Miller.
\newblock Geometric categories and o-minimal structures.
\newblock {\em Duke Math. J.}, 84:497--540, 1996.

\bibitem{Yen2008}
N.~D. Yen, J.-C. Yao, and B.~T. Kien.
\newblock Covering properties at positive-order rates of multifunctions and
  some related topics.
\newblock {\em J. Math. Anal. Appl.}, 338(1):467--478, 2008.

\bibitem{Zheng2015}
X.~Y. Zheng and K.~F. Ng.
\newblock H\"older stable minimizers, tilt stability, and {{H}}\"older metric
  regularity of subdifferentials.
\newblock {\em SIAM J. Optim.}, 25(1):416--438, 2015.

\end{thebibliography}

\end{document}